\documentclass[11pt,a4paper,reqno,english]{ijocta}

\usepackage[dvips]{graphicx}
\usepackage[margin=2cm]{geometry}
\usepackage[font=small,labelfont=bf,tableposition=top]{caption}
\usepackage[font=footnotesize]{subcaption}
\usepackage{multicol}
\usepackage{babel}
\usepackage{amssymb,amsmath,amsfonts}
\usepackage{multirow,url}
\usepackage{blindtext} 

\newtheorem{theorem}{Theorem}

\newtheorem{algorithm}{Algorithm}

\newtheorem{corollary}{Corollary}

\newtheorem{definition}{Definition}

\newtheorem{lemma}{Lemma}

\newtheorem{remark}{Remark}

\begin{document}

\title[\small{\textit{Hybrid algorithms without the extra-steps for equilibrium problems}}]{\textbf{Hybrid algorithms without the extra-steps for equilibrium problems}}
\author[\small{\textit{D. V. Hieu\\/Vol.?, No.?, pp.?-? (2015) \copyright IJOCTA}}]{\normalsize Dang Van Hieu$^a$\vspace{0.58cm} 
\\ \small
$^a$Department of Mathematics, Vietnam National University, Vietnam \\
334 Nguyen Trai Street, Thanh Xuan, Hanoi, Vietnam.\\
Email: dv.hieu83@gmail.com \vspace{0.1cm} \\
\textit{(Received xx $xx, 2015$; in final form xx $xx, 2015$)}}
  
\thanks{Corresponding Author. Email: dv.hieu83@gmail.com}

\begin{abstract}
In this paper, we introduce some new hybrid algorithms for finding a solution of a system of equilibrium problems.
In these algorithms, by constructing specially cutting-halfspaces, we avoid using the extra-steps as in the extragradient method and the 
Armijo linesearch method which are inherently costly when the feasible set has a complex structure. The strong convergence of 
the algorithms is established.

\vspace{7pt}

\vspace{2pt}
\end{abstract}

\maketitle

\section{Introduction}\label{intro}
The paper focuses on the problem of finding a common solution of a system of equilibrium problems (CSEP), i.e., find $x^*\in C$ such that 
\begin{equation}\label{CSEP}
f_i(x^*,y)\ge 0,~\forall y\in C,~i=1,\ldots,N,
\end{equation}
where $f_i:C\times C\to \Re$ are bifunctions. The CSEP is very general in the sense that it includes, as special cases, many mathematical models as: 
convex feasibility problems, common fixed point problems, common solutions to variational inequality problems (CSVIP), common minimizer problems, 
common saddle point problems, variational inequalities 
over the intersection of closed convex subsets, common solutions of operator equations, 
see \cite{ABH2014,AH2014,AH2014b,BO1994,CH2005,CGRS2012,DGM2003,H2015,H2015a} and the references therein. 
In this paper, we introduce two hybrid algorithms for solving CSEPs, namely:
\begin{algorithm}\label{H2015a}
$$
\left \{
\begin{array}{ll}
y^i_{n+1}=  \underset{y\in C}{\rm argmin} \{ \lambda f_i(y^i_n, y) +\frac{1}{2}||x_n-y||^2\},\\
x_{n+1}=P_{C_n\cap Q_n}(x_0),
\end{array}
\right.
$$
\end{algorithm}
\noindent where $C_n\cap Q_n$ is the intersection of $N+1$ specially constructed halfspaces (see, Algorithm \ref{algor1} in Section \ref{main} below) and 
its modification 
\begin{algorithm}\label{H2015b}
$$
\left \{
\begin{array}{ll}
y^i_{n+1}=  \underset{y\in C}{\rm argmin} \{ \lambda f_i(\bar{y}_n, y) +\frac{1}{2}||x_n-y||^2\},\\
\bar{y}_{n+1}=\arg\max\left\{||y^i_{n+1}-x_n||:1\le i\le N\right\},\\
x_{n+1}=P_{C_n\cap Q_n}(x_0),
\end{array}
\right.
$$
\end{algorithm}
\noindent where $C_n,Q_n$ are two specially constructed half-spaces (see Algorithm $\ref{algor2}$ in Section $\ref{main}$ below). 

When CSEP consists of a single equilibrium problem (EP) for a bifunction $f:C\times C\to \Re$, i.e., find $x^*\in C$ such that
\begin{equation}\label{EP}
f(x^*,y)\ge 0,~\forall y\in C,
\end{equation}
then Algorithms \ref{H2015a} and \ref{H2015b} become the following hybrid algorithm.
\begin{algorithm}\label{H2015c}
$$
\left \{
\begin{array}{ll}
y_{n+1}=  \underset{y\in C}{\rm argmin} \{ \lambda f(y_n, y) +\frac{1}{2}||x_n-y||^2\},\\
x_{n+1}=P_{C_n\cap Q_n}(x_0),
\end{array}
\right.
$$
\end{algorithm}
\noindent where $C_n \cap Q_n$ is the intersection of two half-spaces. One of the most popular methods for monotone EPs is the 
proximal point method (PPM) \cite{K2000,M1970,R1976,TT2007,TandaTaka2007} in which solution approximations to EPs are based on the resolvent 
$J_r^f$ \cite{CH2005} of equilibrium bifunction $f$ as
$$J_r^f(x)=\left\{z\in H: f(z,y)+\frac{1}{r}\left\langle y-z,z-x \right\rangle \ge 0, ~ \forall y\in C\right\},$$
where $r>0$ and $x\in H$. For instance, in 2010, C. Jaiboon and P. Kumam \cite{JK2010} proposed the following hybrid algorithm for an EP 
\begin{equation}\label{JK2010}
\left \{
\begin{array}{ll}
u_n ~s.t~f(u_n,y)+\frac{1}{r}\left\langle y-u_n,u_n-x_n \right\rangle \ge 0, ~ \forall y\in C,\\
C_n=\left\{z\in C:||u_n-z||^2\leq ||x_n-z||^2\right\},\\
Q_n=\left\{z\in C: \left\langle x_0-x_n,z-x_n\right\rangle\le 0\right\},\\
x_{n+1}=P_{C_n\cap Q_n}(x_0).
\end{array}
\right.
\end{equation}
They proved that the sequence $\left\{x_n\right\}$ generated by $(\ref{JK2010})$ converges strongly to the projection of $x_0$ onto 
the solution set $EP(f,C)$ of EP. Note that when $f$ is monotone then $J_r^f$ is single-valued, 
strongly monotone and firmly nonexpansive (so nonexpansive). In 2008, the extragradient method \cite{K1976} for variational inequalities 
was extended to EPs in Euclidean spaces by the authors in \cite{QMH2008}. 
The extended extragradient method consists of solving two programs
\begin{equation}\label{QMH2008}
\left \{
\begin{array}{ll}
y_n=  \underset{y\in C}{\rm argmin} \{ \lambda f(x_n, y) +\frac{1}{2}||x_n-y||^2\},\\ 
z_n = \underset{y\in C}{\rm argmin}\{ \lambda f(y_n, y) +\frac{1}{2}||x_n-y||^2\}.\\
\end{array}
\right.
\end{equation}
After that it was further extended to Hilbert spaces \cite{A2011,A2012,DHM2014,H2015a,VSH2013,SVH2011}. In Hilbert spaces, 
the iterative process $(\ref{QMH2008})$, in general, is weakly convergent. 
In 2013, the authors 
in \cite{VSH2013} proposed the following hybrid algorithm for pseudomonotone EPs 
(see Algorithm 2 and Theorem 4.2 in \cite{VSH2013} with 
$T_n$ being the identity operator $I$ for all $n$).
\begin{equation}\label{VSH2013a}
\left \{
\begin{array}{ll}
x_0\in C,~C_0=C,\\
y_n=  \underset{y\in C}{\rm argmin} \{ \lambda f(x_n, y) +\frac{1}{2}||x_n-y||^2\},\\ 
z_n = \underset{y\in C}{\rm argmin}\{ \lambda f(y_n, y) +\frac{1}{2}||x_n-y||^2\},\\
C_{n}=\left\{z\in C:||z_n-z||^2\le||x_n-z||^2\right\},\\
Q_n=\left\{z\in C:\left\langle x_0-x_n,z-x_n\right\rangle\le 0\right\},\\
x_{n+1}=P_{C_n\cap Q_n}(x_0).
\end{array}
\right.
\end{equation}
And they proved the sequence $\left\{x_n\right\}$ generated by (\ref{VSH2013a}) converges strongly to $P_{EP(f,C)}(x_0)$. 
We see that in the extended extragradient methods \cite{A2011,A2012,DHM2014,H2015a,QMH2008,VSH2013,SVH2011}, 
two convex optimization programs must be solved at each iterative step. This seems to be costly if the feasible set $C$ has a 
complex structure and can affect seriously the efficiency of the used method. To avoid solving the second optimization program 
in the extended extragradient method and to reduce conditions posed on the equilibrium bifunction $f$, in 2014, the authors in 
\cite{DHM2014} replaced it by the Armijo linesearch technique and a projection on $C$, and introduced the following strongly convergence hybrid algorithm
\begin{equation}\label{DHM2014}
\left \{
\begin{array}{ll}
x_0\in C,\\
y_n=  \underset{y\in C}{\rm argmin} \{ \lambda f(x_n, y) +\frac{1}{2}||x_n-y||^2\},\\ 
m_n~ \mbox{\rm is the smallest integer number such that}\\
f( (1-\eta^{m_n})x_n+\eta^{m_n}y_n,y_n )+\frac{1}{2\lambda}||x_n-y_n||^2\le 0,\\
u_n=P_C(x_n-\sigma_n g_n),\\
C_{n}=\left\{z\in C:||u_n-z||^2\le||x_n-z||^2\right\},\\
Q_n=\left\{z\in C:\left\langle x_0-x_n,z-x_n\right\rangle\le 0\right\},\\
x_{n+1}=P_{C_n\cap Q_n}(x_0),
\end{array}
\right.
\end{equation}
where $g_n\in \partial f_2(z_n,z_n),z_n= (1-\eta^{m_n})x_n+\eta^{m_n}y_n$ and $\sigma_n=-\eta^{m_n}f(z_n,y_n)/(1-\eta^{m_n})||g_n||^2$. 
However, in the iterative algorithm (\ref{DHM2014}) we still have to perform the linesearch process for finding the smallest integer number $m_n$ 
and a projection onto $C$ for finding $u_n$. 

It is clear that Algorithm \ref{H2015c} does not use the PPM and the resolvent of the equilibrium bifunction $f$. Contrary to the extended 
extragradient method and Armijo linesearch method, in Algorithm \ref{H2015c}, only a convex optimization program is solved at each iterative 
step without the extra-steps. Moreover, the sets $C_n$ and $Q_n$ in Algorithm \ref{H2015c} (also, in Algorithm \ref{H2015b}) are two 
halfspaces. Thus, the next iterate $x_{n+1}=P_{C_n\cap Q_n}(x_0)$ can be expressed by an explicit formula, for instance, in \cite{CH2005,SS2000} 
while the sets $C_n$ and $Q_n$ in the iterative processes (\ref{JK2010}), (\ref{VSH2013a}) 
and (\ref{DHM2014}) deal with the feasible set $C$, and so $x_{n+1}$, in general, is not easy to find.

Some iterative algorithms for solving CSEPs can be found in \cite{AH2014,CH2005,H2015,H2015a} and the references therein. In this paper, 
motivated by the results in \cite{CGRS2012,MS2015} for variational inequalities, we propose Algorithms \ref{H2015a} and \ref{H2015b} for 
finding a common solution of equilibrium problems. We are emphasized that Algorithm \ref{H2015a} seems to be difficult to develop numerical 
methods because in order to find 
the next iterate $x_{n+1}$ we must solve a distance optimization program on the intersection of $N+1$ halfspaces. It is not easy to implement when 
the number of subproblems $N$ is large. However, whether it can help in the design and analysis of more practical algorithms remains to be seen. 
To overcome this complexity, we use a technical extension in \cite{AH2014b,H2015} and introduce Algorithm \ref{H2015b} as 
a modification of Algorithm $\ref{H2015a}$. In this 
algorithm, as mentioned above, the next approximation $x_{n+1}$ is defined as the projection on the intersection of two halfspaces and it is easily 
found by an explicit formula in \cite{CH2005,SS2000}.
\section{Preliminaries}\label{pre}
In this section, we recall some definitions and preliminary results used in the paper. We begin with several concepts of the monotonicity of 
a bifunction and a operator (see, for instance \cite{BO1994,MO1992}).
\begin{definition} A bifunction $f:C\times C\to \Re$ is said to be
\begin{itemize}
\item [$\rm i.$] strongly monotone on $C$ if there exists a constant $\gamma>0$ such that
$$ f(x,y)+f(y,x)\le -\gamma ||x-y||^2,~\forall x,y\in C; $$
\item [$\rm ii.$] monotone on $C$ if 
$$ f(x,y)+f(y,x)\le 0,~\forall x,y\in C; $$
\item [$\rm iii.$] pseudomonotone on $C$ if 
$$ f(x,y)\ge 0 \Longrightarrow f(y,x)\le 0,~\forall x,y\in C;$$
\item [$\rm iv.$] Lipschitz-type continuous on $C$ if there exist two positive constants $c_1,c_2$ such that
$$ f(x,y) + f(y,z) \geq f(x,z) - c_1||x-y||^2 - c_2||y-z||^2, ~ \forall x,y,z \in C.$$
\end{itemize}
\end{definition}
From the definitions above, it is clear that  $i.\Longrightarrow ii. \Longrightarrow iii.$ 
\begin{definition}
An operator $A:C \to H$ is said to be
\begin{itemize}
\item [$\rm i.$] monotone on $C$ if $\left\langle A(x)-A(y),x-y\right\rangle\ge 0$, for all $x,y\in C$;
\item [$\rm ii.$] pseudomonotone on $C$ if
$$\left\langle A(x)-A(y),x-y\right\rangle\ge 0\Longrightarrow \left\langle A(y)-A(x),y-x\right\rangle\le 0$$
for all $x,y\in C$;
\item [$\rm iii.$] $L$ - Lipschitz continuous on $C$ if there exists a positive constant $L$ such that $||A(x)-A(y)||\le L||x-y||$ for all $x,y\in C$. 
\end{itemize}
\end{definition}
For solving CSEP $(\ref{CSEP})$ and EP $(\ref{EP})$, we assume that the 
bifunction $f:C\times C\to \Re$ satisfies the following conditions:
\begin{itemize}
\item[(A1).] $f$ is pseudomonotone on $C$ and $f(x,x)=0$ for all $x\in C$;
\item [(A2).]  $f$ is Lipschitz-type continuous on $C$;
\item [(A3).]   $f$ is weakly continuous on $C$;
\item [(A4).]  $f(x,.)$ is convex and subdifferentiable on $C$ for each $x\in C$.
\end{itemize}  
It is easy to show that under the assumptions $\rm (A1)-A(4)$ then the solution set $EP(f,C)$ of EP (\ref{EP}) is closed and convex 
(see, for instance \cite{QMH2008}). Thus, the solution set $F=\cap_{i=1}^N EP(f_i,C)$ of CSEP (\ref{CSEP}) is also closed and convex. 
In this paper we assume that the solution set $F$ is nonempty.

The metric projection $P_C:H\to C$ is defined by
$$
P_C(x)=\arg\min\left\{\left\|y-x\right\|:y\in C\right\}.
$$ 
Since $C$ is nonempty, closed and convex, $P_C(x)$ exists and is unique. It is also known that $P_C$ has the following characteristic properties, 
see \cite{GR1984} for more details.
\begin{lemma}\label{lem.PropertyPC}
Let $P_C:H\to C$ be the metric projection from $H$ onto $C$. Then
\begin{itemize}
\item [$\rm i.$] $P_C$ is firmly nonexpansive, i.e.,
\begin{equation*}\label{eq:FirmlyNonexpOfPC}
\left\langle P_C x-P_C y,x-y \right\rangle \ge \left\|P_C x-P_C y\right\|^2,~\forall x,y\in H.
\end{equation*}
\item [$\rm ii.$] For all $x\in C, y\in H$,
\begin{equation}\label{eq:ProperOfPC}
\left\|x-P_C y\right\|^2+\left\|P_C y-y\right\|^2\le \left\|x-y\right\|^2.
\end{equation}
\item [$\rm iii.$] $z=P_C x$ if and only if 
\begin{equation}\label{eq:EquivalentPC}
\left\langle x-z,z-y \right\rangle \ge 0,\quad \forall y\in C.
\end{equation}
\end{itemize}
\end{lemma}
%
%
The following lemma is an infinite version of Theorem 27.4 in \cite{R1970} and it is similarly proved.
\begin{lemma}\label{lem.Equivalent_MinPro}
Let $C$ be a convex subset of a real Hilbert space H and $g:C\to \Re$ be a convex and subdifferentiable function on $C$. Then, 
$x^*$ is a solution to the following convex optimization problem
\begin{equation*}\min\left\{g(x):x\in C\right\}
\end{equation*}
if and only if  ~  $0\in \partial g(x^*)+N_C(x^*)$, where $\partial g(.)$ denotes the subdifferential of $g$ and $N_C(x^*)$ is the normal cone 
of  $C$ at $x^*$.
\end{lemma}
We have the following technical lemma.
\begin{lemma}\cite{MS2015}\label{lem.technique}
Let $\left\{M_n\right\}$, $\left\{N_n\right\}$, $\left\{P_n\right\}$ be nonnegative real sequences, 
$\alpha,\beta \in \Re$ and for all $n\ge 0$ the following inequality holds
$$ M_n\le N_n+\beta P_n-\alpha P_{n+1}. $$
If $\sum_{n=0}^\infty N_n <+\infty$ and $\alpha>\beta\ge 0$ then $\lim_{n\to\infty}M_n=0$.
\end{lemma}
\section{Convergence analysis}\label{main}
\setcounter{theorem}{0}
\setcounter{remark}{0}
\setcounter{corollary}{0}
\setcounter{algorithm}{0}
In this section, we rewrite our algorithms for more details and analyze their convergence.
\begin{algorithm}\label{algor1} \textbf{Initialization.} Chose $x_0=x_1 \in H, ~{y}^i_0={y}^i_1\in C$ and set $C_0=Q_0=H$. 
The parameters $\lambda,k$ satisfy the following conditions
$$0< \lambda <\frac{1}{2(c_1+c_2)},~ k>\frac{1}{1-2\lambda(c_1+c_2)}.$$
\textbf{Step 1.} Solve strongly convex optimization programs
 $$y^i_{n+1}=  \underset{y\in C}{\rm argmin} \{ \lambda f_i({y}^i_n, y) +\frac{1}{2}||x_n-y||^2\},~i=1,\ldots,N.$$
\textbf{Step 2.} Compute $x_{n+1}=P_{C_n\cap Q_n}(x_0),$
where $C_n=\cap_{i=1}^N C_n^i$ and
\begin{eqnarray*}
&&C_n^i=\left\{z\in H:||{y}^i_{n+1}-z||^2\leq ||x_n-z||^2+\epsilon_n^i \right\},\\
&&Q_n=\left\{z\in H: \left\langle x_0-x_n,z-x_n\right\rangle\le 0\right\},
\end{eqnarray*}
and $\epsilon_n^i=k||x_n-x_{n-1}||^2+2\lambda c_1||{y}^i_n-{y}^i_{n-1}||^2-(1-\frac{1}{k}-2\lambda c_2)||{y}^i_{n+1}-{y}^i_{n}||^2$. 
Set $n:=n+1$ and go back \textbf{Step 1.}
\end{algorithm}
We have the following central lemma.
\begin{lemma}\label{lem1}
Let $\left\{x_n\right\},\left\{y_n^i\right\}$ be the sequences generated by Algorithm $\ref{algor1}$. 
Then, there hold the following relations for all $i$ and $n$
\begin{itemize}
\item [$\rm i.$] $\left\langle y_{n+1}^i-x_n,y-y_{n+1}^i\right\rangle\ge \lambda\left(f_i({y}^i_n,y_{n+1}^i)-f_i({y}^i_n,y)\right),~\forall y\in C.$
\item [$\rm ii.$]  $||y_{n+1}^i - x^*||^2\leq ||x_n-x^*||^2+\epsilon_n^i$ for all $x^*\in F$.
\end{itemize}
\end{lemma}
\begin{proof}
i. Lemma $\ref{lem.Equivalent_MinPro}$ and
$$
y^i_{n+1}=  \underset{y\in C}{\rm argmin} \{ \lambda f_i({y}^i_n, y) +\frac{1}{2}||x_n-y||^2\}
$$
imply that
\begin{equation*}
0\in \partial_2\left(\lambda f_i(y^i_n,y)+\frac{1}{2}||x_n-y||^2\right)(y^i_{n+1})+N_C(y^i_{n+1}).
\end{equation*}
Therefore, from $\partial \left(||x_n-.||^2\right)=2(.-x_n)$, one obtains
\begin{equation*}
\lambda w+y^i_{n+1}-x_n+\bar{w}=0,
\end{equation*}
where $w\in \partial_2 f_i(y^i_n,y^i_{n+1}):=\partial f_i(y^i_n,.)(y^i_{n+1})$ and $\bar{w}\in N_C(y^i_{n+1})$. From the last equality,
\begin{equation*}
\left\langle y^i_{n+1}-x_n,y-y^i_{n+1}\right\rangle=\lambda \left\langle w,y^i_{n+1}-y\right\rangle+\left\langle \bar{w},y^i_{n+1}-y\right\rangle,~\forall y\in C
\end{equation*}
which implies that
\begin{equation}\label{eq:t}
\left\langle y^i_{n+1}-x_n,y-y^i_{n+1}\right\rangle\ge\lambda \left\langle w,y^i_{n+1}-y\right\rangle,~\forall y\in C
\end{equation}
because of the definition of $N_C$. Since $w\in \partial_2 f(y^i_n,y^i_{n+1})$,
\begin{equation*}
f_i(y^i_n,y)-f_i(y^i_n,y^i_{n+1})\ge \left\langle w,y-y^i_{n+1}\right\rangle,~\forall y\in C.
\end{equation*}
Thus, 
$$ 
\left\langle w,y^i_{n+1}-y\right\rangle\ge f_i(y^i_n,y^i_{n+1})-f_i(y^i_n,y),~\forall y\in C.
 $$
This together with the relation $(\ref{eq:t})$ implies that
\begin{equation}\label{eq:1}
\left\langle y^i_{n+1}-x_n,y-y^i_{n+1}\right\rangle\ge \lambda\left(f_i(y^i_n,y^i_{n+1})-f_i(y^i_n,y)\right),~\forall y\in C.
\end{equation}
ii. From Lemma $\ref{lem1}$.i., we have
\begin{equation}\label{eq:2}
\left\langle y^i_{n}-x_{n-1},y-y^i_{n}\right\rangle\ge \lambda\left(f_i(y^i_{n-1},y^i_{n})-f_i(y^i_{n-1},y)\right),~\forall y\in C.
\end{equation}
Substituting $y=y^i_{n+1}\in C$ onto $(\ref{eq:2})$, we obtain
$$
\left\langle y^i_{n}-x_{n-1},y^i_{n+1}-y^i_{n}\right\rangle\ge \lambda\left(f_i(y^i_{n-1},y^i_{n})-f_i(y^i_{n-1},y^i_{n+1})\right).
$$
 Thus,
\begin{equation}\label{eq:3}
\lambda\left(f_i(y^i_{n-1},y^i_{n+1})-f_i(y^i_{n-1},y^i_{n})\right)\ge\left\langle y^i_{n}-x_{n-1},y^i_{n}-y^i_{n+1}\right\rangle .
\end{equation}
Substituting $y=x^*$ onto $(\ref{eq:1})$ we also obtain
\begin{equation}\label{eq:3*}
\left\langle y^i_{n+1}-x_n,x^*-y^i_{n+1}\right\rangle\ge \lambda\left(f_i(y^i_n,y^i_{n+1})-f_i(y^i_n,x^*)\right).
\end{equation}
Since $x^*\in EP(f_i,C)$ and $y^i_n\in C$, $f_i(x^*,y^i_n)\ge 0$. Hence, $f_i(y^i_n,x^*)\le 0$ because of the pseudomonotonicity of $f_i$. 
This together with $(\ref{eq:3*})$ implies that 
\begin{equation}\label{eq:4}
\left\langle y^i_{n+1}-x_n,x^*-y^i_{n+1}\right\rangle\ge \lambda f_i(y^i_n,y^i_{n+1}).
\end{equation}
Using the Lipschitz-type continuity of $f_i$ with $x=y^i_{n-1}$, $y=y^i_{n}$ and $z=y^i_{n+1}$, we get
\begin{equation*}
f_i(y^i_{n-1},y^i_n)+f_i(y^i_n,y^i_{n+1})\ge f_i(y^i_{n-1},y^i_{n+1})-c_1||y^i_{n-1}-y^i_n||^2-c_2||y^i_n-y^i_{n+1}||^2
\end{equation*}
which implies that
\begin{equation}\label{eq:5}
f_i(y^i_n,y^i_{n+1})\ge f_i(y^i_{n-1},y^i_{n+1})-f_i(y^i_{n-1},y^i_n)-c_1||y^i_{n-1}-y^i_n||^2-c_2||y^i_n-y^i_{n+1}||^2.
\end{equation}
Combining $(\ref{eq:4})$ and $(\ref{eq:5})$, we see that
\begin{eqnarray*}
\left\langle y^i_{n+1}-x_n,x^*-y^i_{n+1}\right\rangle&\ge& \lambda\left\{f_i(y^i_{n-1},y^i_{n+1})-f_i(y^i_{n-1},y^i_n)\right\}\nonumber\\
&&-\lambda c_1||y^i_{n-1}-y^i_n||^2-\lambda c_2||y^i_n-y^i_{n+1}||^2.\label{eq:6}
\end{eqnarray*}
From this and the relation $(\ref{eq:3})$,
\begin{eqnarray*}
\left\langle y^i_{n+1}-x_n,x^*-y^i_{n+1}\right\rangle&\ge& \left\langle y^i_{n}-x_{n-1},y^i_{n}-y^i_{n+1}\right\rangle-\lambda c_1||y^i_{n-1}-y^i_n||^2\nonumber\\
&&-\lambda c_2||y^i_n-y^i_{n+1}||^2.
\end{eqnarray*}
Multiplying both the sides of the last inequality by 2, we obtain
\begin{eqnarray}\label{eq:7}
2\left\langle y^i_{n+1}-x_n,x^*-y^i_{n+1}\right\rangle&-& 2\left\langle y^i_{n}-x_{n-1},y^i_{n}-y^i_{n+1}\right\rangle\ge-2\lambda c_1||y^i_{n-1}-y^i_n||^2\nonumber\\
&&-2\lambda c_2||y^i_n-y^i_{n+1}||^2.
\end{eqnarray}
We have the following fact
\begin{eqnarray}
&&2\left\langle y^i_{n+1}-x_n,x^*-y^i_{n+1}\right\rangle=||x_n-x^*||^2-||y^i_{n+1}-x^*||^2-||x_n-y^i_{n+1}||^2\nonumber\\ 
&=& ||x_n-x^*||^2-||y^i_{n+1}-x^*||^2-||x_n-x_{n-1}||^2-2\left\langle x_n-x_{n-1},x_{n-1}-y^i_{n+1}\right\rangle\nonumber\\
&&-||x_{n-1}-y^i_{n+1}||^2\nonumber\\
&=&||x_n-x^*||^2-||y^i_{n+1}-x^*||^2-||x_n-x_{n-1}||^2-2\left\langle x_n-x_{n-1},x_{n-1}-y^i_{n+1}\right\rangle\nonumber\\
&&-||x_{n-1}-y^i_{n}||^2-2\left\langle x_{n-1}-y^i_{n},y^i_{n}-y^i_{n+1}\right\rangle-||y^i_{n}-y^i_{n+1}||^2.\label{eq:8**}
\end{eqnarray}
We also have
\begin{eqnarray}
&-&2\left\langle x_n-x_{n-1},x_{n-1}-y^i_{n+1}\right\rangle\le2||x_n-x_{n-1}||||x_{n-1}-y^i_{n+1}||\nonumber\\
&&\le2||x_n-x_{n-1}||||x_{n-1}-y^i_{n}||+2||x_n-x_{n-1}||||y^i_{n}-y^i_{n+1}||\nonumber\\ 
&&\le|| x_n-x_{n-1}||^2+||x_{n-1}-y^i_{n}||^2+k|| x_n-x_{n-1}||^2+\frac{1}{k}||y^i_{n}-y^i_{n+1}||^2 \label{eq:8t}
\end{eqnarray}
in which the first inequality is followed from the Cauchy-Schwarz inequality, the second inequality is followed from the triangle 
inequality and the last inequality is true by the Cauchy inequality $2ab\le a^2+b^2$. From the relations (\ref{eq:8**}) and (\ref{eq:8t}),
\begin{align*}
&2\left\langle y^i_{n+1}-x_n,x^*-y^i_{n+1}\right\rangle\le ||x_n-x^*||^2-||y^i_{n+1}-x^*||^2+k|| x_n-x_{n-1}||^2\nonumber\\
&+2\left\langle y^i_n-x_{n-1},y^i_{n}-y^i_{n+1}\right\rangle+\left(\frac{1}{k}-1\right)||y^i_{n}-y^i_{n+1}||^2\label{eq:8***}.
\end{align*}
Thus, 
\begin{eqnarray*}
&&2\left\langle y^i_{n+1}-x_n,x^*-y^i_{n+1}\right\rangle-2\left\langle y^i_n-x_{n-1},y^i_{n}-y^i_{n+1}\right\rangle\le||x_n-x^*||^2\nonumber\\
&&-||y^i_{n+1}-x^*||^2+k|| x_n-x_{n-1}||^2+\left(\frac{1}{k}-1\right)||y^i_{n}-y^i_{n+1}||^2\label{eq:8***}
\end{eqnarray*}
which together with $(\ref{eq:7})$ leads to
\begin{eqnarray*}
-2\lambda c_1||y^i_{n-1}-y^i_n||^2&-&2\lambda c_2||y^i_n-y^i_{n+1}||^2\le||x_n-x^*||^2-||y^i_{n+1}-x^*||^2\\
&+&k|| x_n-x_{n-1}||^2+\left(\frac{1}{k}-1\right)||y^i_{n}-y^i_{n+1}||^2. 
\end{eqnarray*}
Hence,
\begin{eqnarray*}
||y^i_{n+1}-x^*||^2&\leq&||x_n-x^*||^2+k|| x_n-x_{n-1}||^2+2\lambda c_1||y^i_{n-1}-y^i_n||^2\\
&&-\left(1-\frac{1}{k}-2\lambda c_2\right)||y^i_{n}-y^i_{n+1}||^2. \\
&=&||x_n-x^*||^2+\epsilon^i_n
\end{eqnarray*}
in which the last equality is followed from the definition of $\epsilon_n^i$ in Step 2 of Algorithm \ref{algor1}. Lemma $\ref{lem1}$ is proved.
\end{proof}
\begin{lemma}\label{lem4}
Let $\left\{x_n\right\},\left\{y^i_n\right\}$ be the sequences generated by Algorithm $\ref{algor1}$. Then, there hold the following relations
\begin{itemize}
\item [$\rm i.$] $F\subset C_n\cap Q_n$ for all $n\ge 0$.
\item [$\rm ii.$] $\lim\limits_{n\to\infty}||x_{n+1}-x_n||=\lim\limits_{n\to\infty}||y^i_n-x_n||=\lim\limits_{n\to\infty}||y^i_{n+1} - y^i_{n}||=0.$
\end{itemize}
\end{lemma}
\begin{proof}
i. Since the sets $C_n^i$ are the halfspaces, $C_n=\cap_{i=1}^N C_n^i$ and $Q_n$ are closed and convex. Lemma $\ref{lem1}$.ii. and the definition of $C_n$ ensure that $F\subset C_n$ for all $n\ge 0$. 
It is clear that $F\subset H= C_0\cap Q_0$. Assume that $F\subset C_n\cap Q_n$ for some $n\ge 0$. From $x_{n+1}=P_{C_n\cap Q_n}(x_0)$ 
and Lemma $\ref{lem.PropertyPC}$.iii., we see that $\left\langle z-x_{n+1},x_0-x_{n+1}\right\rangle\le 0$ for all $z\in C_n\cap Q_n$. Since $F\subset C_n\cap Q_n$, 
$$\left\langle z-x_{n+1},x_0-x_{n+1}\right\rangle\le 0$$
for all $z\in F$. Thus, $F\subset Q_{n+1}$ because of the definition of $Q_{n+1}$ or 
$F\subset C_{n+1}\cap Q_{n+1}$. By the induction, $F\subset C_n\cap Q_n$ for all $n\ge 0$. Since $F$ is nonempty, $C_n\cap Q_n$ is too. 
Hence, $P_F(x_0)$ and $P_{C_n\cap Q_n}(x_0)$ are well-defined.\\
ii. From the definition of $Q_n$ and Lemma $\ref{lem.PropertyPC}$.iii., we see that $x_n=P_{Q_n}(x_0)$. Thus, by Lemma $\ref{lem.PropertyPC}$.ii 
we have
\begin{equation}\label{eq:8}
||z-x_n||^2\le||z-x_0||^2-||x_n-x_0||^2,~\forall z\in Q_n.
\end{equation}
Substituting $z=x^\dagger:=P_{F}(x_0)\in Q_n$ onto $(\ref{eq:8})$, one has 
\begin{equation}\label{eq:8*}
||x^\dagger-x_0||^2-||x_n-x_0||^2\ge ||x^\dagger-x_n||^2\ge 0.
\end{equation}
Hence, $\left\{||x_n-x_0||\right\}$ is a bounded sequence, and so is $\left\{x_n\right\}$. Substituting $z=x_{n+1}\in Q_n$ onto $(\ref{eq:8})$, one also has 
\begin{equation}\label{eq:9}
0\le||x_{n+1}-x_n||^2\le||x_{n+1}-x_0||^2-||x_n-x_0||^2.
\end{equation}
This implies that $\left\{||x_n-x_0||\right\}$ is non-decreasing. Hence, there exists the limit of $\left\{||x_n-x_0||\right\}$. 
By $(\ref{eq:9})$, 
$$
\sum_{n=1}^K||x_{n+1}-x_n||^2\le ||x_{K+1}-x_0||^2-||x_1-x_0||^2,~\forall K\ge 1.
$$
Passing the limit in the last inequality as $K\to\infty$, we obtain
\begin{equation}\label{eq:10*}
\sum_{n=1}^\infty||x_{n+1}-x_n||^2<+\infty.
\end{equation}
Thus,
\begin{equation}\label{eq:11*}
\lim_{n\to\infty}||x_{n+1}-x_n||=0.
\end{equation}
Since $x_{n+1}\in C_n=\cap_{i=1}^N C_n^i$, $x_{n+1}\in C^i_n$. From the definition of $C_n^i$,
\begin{equation}\label{eq:12}
 ||y^i_{n+1} - x_{n+1}||^2\leq ||x_n-x_{n+1}||^2+\epsilon^i_n.
\end{equation}
Set $M_n=||y^i_{n+1} - x_{n+1}||^2$, $N_n=||x_n-x_{n+1}||^2+k||x_n-x_{n-1}||^2$, $P_n=||y^i_n-y^i_{n-1}||^2$, $\beta=2\lambda c_2$, 
and  $\alpha=1-\frac{1}{k}-2\lambda c_1$. From the definition of $\epsilon^i_n$, $\epsilon^i_n=k||x_n-x_{n-1}||^2+\beta P_n-\alpha P_{n+1}$. 
Thus, from $(\ref{eq:12})$, 
\begin{equation}\label{eq:13*}
M_n\le N_n+\beta P_n-\alpha P_{n+1}.
\end{equation}
By the hypothesises of $\lambda, ~k$ and $(\ref{eq:10*})$, we see that $\alpha>\beta\ge 0$ and $\sum_{n=1}^\infty N_n<+\infty$. Lemma 
$\ref{lem.technique}$ and $(\ref{eq:13*})$ imply that $M_n\to 0$, or 
\begin{equation*}\label{eq:14}
\lim_{n\to\infty}||y^i_{n+1}- x_{n+1}||=0.
\end{equation*}
This together with the relation $(\ref{eq:11*})$ and the inequality $||y^i_{n+1}-y^i_{n}||\le||y^i_{n+1}-x_{n+1}||+||x_{n+1}-x_n||+||x_{n}-y^i_{n}||$ implies that 
\begin{equation*}\label{eq:15}
\lim_{n\to\infty}||y^i_{n+1} - y^i_{n}||=0.
\end{equation*}
In addition, the sequence $\left\{y_{n}\right\}$ is also bounded because of the boundedness of $\left\{x_n\right\}$. Lemma $\ref{lem4}$ is proved.\\
\end{proof}
We have the following main result.
\begin{theorem}\label{theo.1}
Let $C$ be a nonempty closed convex subset of a real Hilbert space $H$. Assume that the bifunctions $\left\{f_i\right\}_{i=1}^N:C\times C\to \Re$ satisfy 
all conditions $\rm (A1)-(A4)$. In addition the solution set $F=\cap_{i=1}^N EP(f_i,C)$ is nonempty. Then, the sequences $\left\{x_n\right\}$, 
$\left\{y^i_n\right\}$ generated by Algorithm $\ref{algor1}$ converge strongly to $P_{F}(x_0)$.
\end{theorem}
\begin{proof}
Assume that $p$ is any weak cluster point of $\left\{x_n\right\}$. 
Without loss of generality, we can write $x_n\rightharpoonup p$ as $n\to \infty$. Since $||x_n-y^i_n||\to 0$, $y^i_n\rightharpoonup p$. 
Next, we show that $p\in F$.  Indeed, from Lemma \ref{lem1}.i., we get 
$$
\left\langle y_{n+1}^i-x_n,y-y_{n+1}^i\right\rangle\ge \lambda\left(f_i({y}^i_n,y_{n+1}^i)-f_i({y}^i_n,y)\right),~\forall y\in C.
$$
Hence,
\begin{equation}\label{eq:15}
\lambda\left(f_i(y^i_n,y)-f_i(y^i_n,y^i_{n+1})\right)\ge \left\langle x_n-y^i_{n+1},y-y^i_{n+1}\right\rangle,~\forall y\in C.
\end{equation}
Passing the limit in $(\ref{eq:15})$ as $n\to\infty$ and using the hypothesis $\rm (A3)$, Lemma $\ref{lem4}$.ii., the bounedness of $\left\{y^i_n\right\}$ 
and $\lambda>0$ we obtain $f_i(p,y)\ge 0$ for all $y\in C$. Thus, $p\in EP(f_i,C)$ for all $i=1,\ldots,N$ or $p\in F$. Finally, we show that 
$\left\{x_n\right\}\to x^\dagger:= P_F(x_0)$ as $n\to\infty$. Indeed, from the inequality $(\ref{eq:8*})$, we get
$$ ||x_n-x_0||\le ||x^\dagger-x_0||.$$
By the weak lower semicontinuity of the norm $||.||$ and $x_n\rightharpoonup p$, we have
\begin{equation*}
||p-x_0||\le \lim_{n\to\infty}\inf||x_{n}-x_0||\le \lim_{n\to\infty}\sup||x_{n}-x_0||\le||x^\dagger-x_0||.
\end{equation*}
By the definition of $x^\dagger$, $p=x^\dagger$ and $\lim_{n\to\infty}||x_{n}-x_0||=||x^\dagger-x_0||$. Thus, $\lim_{n\to\infty}||x_{n}||=||x^\dagger||$. 
By the Kadec-Klee property of the Hilbert space $H$, we have $x_{n}\to x^\dagger=P_{F}(x_0)$ as $n\to\infty$. From Lemma $\ref{lem4}$.ii., we 
also see that $\left\{y^i_n\right\}$ converges strongly to $P_{F}(x_0)$. This completes the proof of Theorem.
\end{proof}
\begin{corollary}\label{cor2}
Let $\left\{A_i\right\}_{i=1}^N:C\to H$ be pseudomonotone and $L$ - Lipschitz continuous operators such that $F=\cap_{i=1}^N VI(A_i,C)$ is nonempty. 
Let $\left\{x_n\right\}$ be the sequence generated by the following manner: $x_0=x_1\in H$, $y_0^i=y_1^i\in C$ and 
\begin{equation*}\label{eq:}
\left \{
\begin{array}{ll}
y^i_{n+1}=P_C(x_n-\lambda A_i(y_n^i)),\\
C_n^i=\left\{z\in H:||{y}^i_{n+1}-z||^2\leq ||x_n-z||^2+\epsilon_n^i \right\},\\
C_n=\cap_{i=1}^N C_n^i,~Q_n=\left\{z\in H: \left\langle x_0-x_n,z-x_n\right\rangle\le 0\right\},\\
x_{n+1}=P_{C_n\cap Q_n}(x_0),
\end{array}
\right.
\end{equation*}
where $\epsilon_n^i=k||x_n-x_{n-1}||^2+\lambda L||{y}^i_n-{y}^i_{n-1}||^2-(1-\frac{1}{k}-\lambda L)||{y}^i_{n+1}-{y}^i_{n}||^2$, 
$\lambda\in (0,\frac{1}{L})$ and $k>\frac{1}{1-\lambda L}$. Then, the sequence $\left\{x_n\right\}$ converges strongly to 
$P_{F}(x_0)$.
\end{corollary}
\begin{proof}
Set $f_i(x,y)=\left\langle A_i(x),y-x\right\rangle$ for all $i=1,\ldots,N$ and $x,y\in C$. According to Algorithm $\ref{algor1}$ we have
\begin{eqnarray*}
y^i_{n+1}&=&\underset{y\in C}{\rm argmin} \{ \lambda \left\langle A_i(y^i_n),y-y^i_n\right\rangle+\frac{1}{2}||x_n-y||^2\}\\ 
&=&\underset{y\in C}{\rm argmin} \{\frac{1}{2}||y-(x_n-\lambda A_i(y^i_n))||^2-\frac{\lambda^2}{2}||A_i(y^i_n)||^2-\lambda \left\langle A_i(y_n),y_n-x_n\right\rangle\}\\
&=&\underset{y\in C}{\rm argmin} \{\frac{1}{2}||y-(x_n-\lambda A_i(y^i_n))||^2\}\\
&=&P_C\left(x_n-\lambda A_i(y^i_n)\right).
\end{eqnarray*}
Thus, Corollary \ref{cor1} is directly followed from Theorem \ref{theo.1}.
\end{proof}
\begin{remark}
Corollary $\ref{cor2}$ can be considered as an improvement of Algorithm 3.1 in \cite{CGRS2012} via the aspect as only a projection onto the feasible set $C$ 
is found at each iterative step while Algorithm 3.1 in \cite{CGRS2012} uses the extragradient method (double projection method) onto $C$.
\end{remark}
In Algorithm $\ref{algor1}$ and Corollary \ref{cor2}, in order to find $x_{n+1}$ we must solve a distance optimization program
onto the intersection of $N+1$ halfspaces
\begin{equation*}\label{eq:}
\left\{
\begin{array}{ll}
\min||z-x_0||^2,\\
\mbox{such that}\quad z\in C_n^1\cap \ldots \cap C_n^N \cap Q_n. 
\end{array}
\right.
\end{equation*}
This can be costly if the number of subproblems $N$ is large. To overcome this complexity, we use a technique extension 
(see, \cite{AH2014,AH2014b,H2015}) and propose the following hybrid algorithm for CSEPs.
\begin{algorithm}\label{algor2} \textbf{Initialization.} Chose $x_0=x_1 \in H, ~\bar{y}_0=\bar{y}_1\in C$ and set $C_0=Q_0=H$. 
The parameters $\lambda,k$ satisfy the following conditions 
$$0< \lambda <\frac{1}{2(c_1+c_2)},~ k>\frac{1}{1-2\lambda(c_1+c_2)}.$$
\textbf{Step 1.} Compute $y^i_{n+1}=  \underset{y\in C}{\rm argmin} \{ \lambda f_i(\bar{y}_n, y) +\frac{1}{2}||x_n-y||^2\},~i=1,\ldots,N$.\\
\textbf{Step 2.} Find  $\bar{y}_{n+1}=\arg\max\left\{||{y}^i_{n+1}-x_n||:1\le i\le N\right\}$.\\
\textbf{Step 3.} Compute $x_{n+1}=P_{C_n\cap Q_n}(x_0),$
where 
\begin{eqnarray*}
&&C_n=\left\{z\in H:||\bar{y}_{n+1}-z||^2\leq ||x_n-z||^2+\epsilon_n \right\},\\
&&Q_n=\left\{z\in H: \left\langle x_0-x_n,z-x_n\right\rangle\le 0\right\},
\end{eqnarray*}
and $\epsilon_n=k||x_n-x_{n-1}||^2+2\lambda c_1||\bar{y}_n-\bar{y}_{n-1}||^2-(1-\frac{1}{k}-2\lambda c_2)||\bar{y}_{n+1}-\bar{y}_{n}||^2$. 
Set $n:=n+1$ and go back \textbf{Step 1.}
\end{algorithm}
\begin{remark}
Additional computation $\bar{y}_{n+1}$ in Step 2 of Algorithm \ref{algor2} is negligible. The set $C_n\cap Q_n$ is the intersection of two halfspaces 
and an explicit formula for $x_{n+1}=P_{C_n\cap Q_n}(x_0)$ in Algorithm \ref{algor2} can be found, for instance, in \cite{CH2005,SS2000}.
\end{remark}
\begin{theorem}\label{theo2}
Theorem $\ref{theo.1}$ remains true for Algorithm $\ref{algor2}$.
\end{theorem}
\begin{proof}
By repeating the proof of Lemma $\ref{lem1}$, we obtain 
\begin{eqnarray}
&&\left\langle y_{n+1}^i-x_n,y-y_{n+1}^i\right\rangle\ge \lambda\left(f_i(\bar{y}_n,y_{n+1}^i)-f_i(\bar{y}_n,y)\right),~\forall y\in C,\nonumber\\
&&||y^i_{n+1} - x^*||^2\leq ||x_n-x^*||^2+\epsilon_n \label{eq:17}
\end{eqnarray}
for all $x^*\in F,~i=1,\ldots,N$. The inequality $(\ref{eq:17})$ implies that
\begin{equation*}\label{eq:18}
||\bar{y}_{n+1} - x^*||^2\leq ||x_n-x^*||^2+\epsilon_n.
\end{equation*}
Thus, by the same arguments as in the proof of Lemma $\ref{lem4}$ we get $F\subset C_n\cap Q_n$ for all $n\ge 0$ and 
\begin{equation}\label{eq:19}
\lim_{n\to\infty}||x_{n+1}-x_n||=\lim_{n\to\infty}||\bar{y}_{n+1}-x_{n+1}||=0.
\end{equation}
From $(\ref{eq:19})$ and the triangle inequality $||\bar{y}_{n+1}-x_{n}||\le ||\bar{y}_{n+1}-x_{n+1}||+||{x}_{n+1}-x_{n}||$, one has 
$$\lim_{n\to\infty}||\bar{y}_{n+1}-x_{n}||=0.$$
From the definition of $\bar{y}_{n+1}$, we obtain
\begin{equation*}\label{eq:20}
\lim_{n\to\infty}||y^i_{n+1}-x_{n}||=0,~\forall i=1,\ldots,N.
\end{equation*}
The rest of the proof of Theorem $\ref{theo2}$ is similar to one of Theorem $\ref{theo.1}.$
\end{proof}
By the same arguments as in the proof of Corollary \ref{cor2} and using Theorem \ref{theo2}, we obtain the following result.
\begin{corollary}\label{cor3}
Let $\left\{A_i\right\}_{i=1}^N:C\to H$ be pseudomonotone and $L$ - Lipschitz continuous operators such that $F=\cap_{i=1}^N VI(A_i,C)$ is nonempty. 
Let $\left\{x_n\right\}$ be the sequence generated by the following manner: $x_0=x_1\in H$, $y_0^i=y_1^i\in C$ and 
\begin{equation*}\label{eq:}
\left \{
\begin{array}{ll}
y^i_{n+1}=P_C(x_n-\lambda A_i(\bar{y}_n)),\\
\bar{y}_{n+1}=\arg\max\left\{||{y}^i_{n+1}-x_n||:1\le i\le N\right\},\\
C_n=\left\{z\in H:||\bar{y}_{n+1}-z||^2\leq ||x_n-z||^2+\epsilon_n \right\},\\
Q_n=\left\{z\in H: \left\langle x_0-x_n,z-x_n\right\rangle\le 0\right\},\\
x_{n+1}=P_{C_n\cap Q_n}(x_0),
\end{array}
\right.
\end{equation*}
where $\epsilon_n=k||x_n-x_{n-1}||^2+\lambda L||\bar{y}_n-\bar{y}_{n-1}||^2-(1-\frac{1}{k}-\lambda L)||\bar{y}_{n+1}-\bar{y}_{n}||^2$, 
$\lambda\in (0,\frac{1}{L})$ and $k>\frac{1}{1-\lambda L}$. Then, the sequence $\left\{x_n\right\}$ converges strongly to 
$P_{F}(x_0)$.
\end{corollary}
In the special case, CSEP consists of an EP for a bifunction $f:C\times C\to \Re$ then Algorithms $\ref{algor1}$ and \ref{algor2} are reduced to 
the following hybrid algorithm.
\begin{algorithm}[The hybrid algorithm without the extra-steps]\label{algor3}
$$ 
\left \{
\begin{array}{ll}
x_0=x_1\in H, ~y_0=y_1\in C,\\
y_{n+1}=\underset{y\in C}{\arg\min}\left\{\lambda f(y_n,y)+\frac{1}{2}||x_n-y||^2\right\},\\
C_n=\left\{z\in H:||{y}_{n+1}-z||^2\leq ||x_n-z||^2+\epsilon_n \right\},\\
Q_n=\left\{z\in H: \left\langle x_0-x_n,z-x_n\right\rangle\le 0\right\},\\
x_{n+1}=P_{C_n\cap Q_n}(x_0),
\end{array}
\right.
 $$
\end{algorithm}
\noindent where $\epsilon_n=k||x_n-x_{n-1}||^2+2\lambda c_1||{y}_n-{y}_{n-1}||^2-(1-\frac{1}{k}-2\lambda c_2)||{y}_{n+1}-{y}_{n}||^2$, 
$\lambda\in (0,\frac{1}{c_1+c_2})$ and $k>\frac{1}{1-\lambda (c_1+c_2)}$.
\begin{remark}
Contrary to the extragradient method and the Armijo linesearch method \cite{A2011,A2012,DHM2014,H2015a,QMH2008,SVH2011,VSH2013} 
for EPs, in Algorithm \ref{algor3} we only need to solve an optimization program at each iterative step without the extra-steps which are inherently costly if the 
feasible set $C$ has a complex structure. Moreover, the sets $C_n, Q_n$ in Algorithm $\ref{algor3}$ are two halfspaces and do not relate to the 
feasible set $C$.
\end{remark}
\begin{theorem}\label{theo.3}
Let $f:C\times C\to \Re$ be a bifunction satisfying all conditions $\rm (A1)-(A4)$ such that the solution set $EP(f,C)$ is nonempty. 
Then, the sequence $\left\{x_n\right\}$ generated by Algorithm \ref{algor3} converges strongly to 
$P_{EP(f,C)}(x_0)$.
\end{theorem}
\begin{proof}
Theorem \ref{theo.3} is directly followed from Theorems \ref{theo.1} and \ref{theo2}.
\end{proof}
\begin{corollary}\cite[Algorithm 1]{MS2015}\label{cor1}
Let $A:C\to H$ be a monotone and $L$ - Lipschitz continuous operator such that $VI(A,C)$ is nonempty. 
Let $\left\{x_n\right\}$ be the sequence generated by the following manner: $x_0=x_1\in H$, $y_0=y_1\in C$ and 
\begin{equation*}\label{eq:}
\left \{
\begin{array}{ll}
y_{n+1}=P_C(x_n-\lambda A(y_n)),\\
C_n=\left\{z\in H:||{y}_{n+1}-z||^2\leq ||x_n-z||^2+\epsilon_n \right\},\\
Q_n=\left\{z\in H: \left\langle x_0-x_n,z-x_n\right\rangle\le 0\right\},\\
x_{n+1}=P_{C_n\cap Q_n}(x_0),
\end{array}
\right.
\end{equation*}
where $\epsilon_n=k||x_n-x_{n-1}||^2+\lambda L||{y}_n-{y}_{n-1}||^2-(1-\frac{1}{k}-\lambda L)||{y}_{n+1}-{y}_{n}||^2$, 
$\lambda\in (0,\frac{1}{L})$ and $k>\frac{1}{1-\lambda L}$. Then, the sequence $\left\{x_n\right\}$ converges strongly to 
$P_{VI(A,C)}(x_0)$.
\end{corollary}
\begin{proof}
Using Theorem \ref{theo.3} for the bifunction $f(x,y)=\left\langle A(x),y-x\right\rangle$ for all $x,y\in C$.
\end{proof}
\begin{remark}
Algorithms $\ref{algor1}$ and $\ref{algor2}$ are the parallel algorithms in the sense that the intermediate approximations 
$y_n^i,~i=1,\ldots,N$ can be simultaneously found at each iterative step and on each subproblem. In fact, we can design the following sequential algorithm 
for CSEP.
\begin{equation}\label{algorSequential}
\left \{
\begin{array}{ll}
x_0=x_1\in H, ~y_0=y_1\in C,\\
y_{n+1}=\underset{y\in C}{\arg\min}\left\{\lambda f_{[n]}(y_n,y)+\frac{1}{2}||x_n-y||^2\right\},\\
C_n=\left\{z\in H:||{y}_{n+1}-z||^2\leq ||x_n-z||^2+\epsilon_n \right\},\\
Q_n=\left\{z\in H: \left\langle x_0-x_n,z-x_n\right\rangle\le 0\right\},\\
x_{n+1}=P_{C_n\cap Q_n}(x_0),
\end{array}
\right.
\end{equation}
where $\epsilon_n=k||x_n-x_{n-1}||^2+2\lambda c_1||{y}_n-{y}_{n-1}||^2-(1-\frac{1}{k}-2\lambda c_2)||{y}_{n+1}-{y}_{n}||^2$, 
$\lambda\in (0,\frac{1}{c_1+c_2})$, $k>\frac{1}{1-\lambda (c_1+c_2)}$ and $[n]=n (mod~N)+1$ is the mod function taking values in 
$\left\{1,2,\ldots,N\right\}$. It is also easy to show that the sequence $\left\{x_n\right\}$ generated by (\ref{algorSequential}) converges 
strongly to $P_F(x_0)$.
\end{remark}


\vspace{1.5cm}

\noindent
\textit{\textbf{Dang Van Hieu} is a researcher at Center for High-Performance Computing and is also a lecturer at Department of Mathematics,
Vietnam National University, Hanoi, Vietnam. His area of research includes optimization theory, equilibrium problems and their applications.}


\begin{thebibliography}{99}

\normalsize

\bibitem{A2011} Anh, P. N., A hybrid extragradient method extended to fixed point problems and equilibrium problems. Optimization,  62 (2), 271--283 (2013).
\bibitem{A2012}Anh, P. N., Strong convergence theorems for nonexpansive
mappings and Ky Fan inequalities. J. Optim. Theory Appl., 154, 303--320 (2012).
\bibitem{ABH2014}Anh, P. K., Buong, Ng., Hieu, D. V., Parallel methods for regularizing systems of equations involving accretive operators. 
Appl. Anal., 93 (10), 2136-2157 (2014).
\bibitem{AH2014}Anh, P.K., Hieu, D.V, Parallel and sequential hybrid methods for a finite family of asymptotically quasi $\phi$-nonexpansive mappings. 
J. Appl. Math. Comput., 48, 241-263 (2015).
\bibitem{AH2014b} Anh, P.K., Hieu, D.V., Parallel hybrid methods for variational inequalities, equilibrium problems and common fixed point 
problems. Vietnam J. Math. (2015), DOI:10.1007/s10013-015-0129-z.
\bibitem{BO1994} Blum, E., Oettli, W., From optimization and variational inequalities to equilibrium problems. Math. Program., 63, 123-145 (1994).
\bibitem{CH2005} Combettes, P. L., Hirstoaga, S. A., Equilibrium programming in Hilbert spaces. J. Nonlinear Convex Anal., 6 (1), 117-136 (2005).
\bibitem{CGRS2012}Censor, Y., Gibali, A., Reich, S., Sabach, S., Common Solutions to Variational Inequalities. 
Set-Valued Var. Anal., 20, 229-247 (2012).
\bibitem{DGM2003} Daniele, P., Giannessi, F., Maugeri, A., Equilibrium problems and variational models, Kluwer, (2003).
\bibitem{DHM2014} Dinh, B.V., Hung, P.G., Muu, L.D., Bilevel optimization as a regularization approach to pseudomonotone equilibrium problems. 
Numer. Funct. Anal.  Optim., 35 (5), 539--563 (2014).
\bibitem{GR1984} Goebel, K., Reich, S., Uniform Convexity, Hyperbolic Geometry, and Nonexpansive Map-pings. Marcel Dekker, New York (1984).
\bibitem{H2015} Hieu, D. V., A parallel hybrid method for equilibrium problems, variational inequalities and nonexpansive mappings in Hilbert space. 
J. Korean Math. Soc., 52, 373-388 (2015).
\bibitem{H2015a} Hieu, D. V.: The common solutions to pseudomonotone equilibrium problems. Bull. Iranian Math. Soc. (2015) 
(accepted for publication).
\bibitem{K2000} Konnov, I.V., Combined relaxation methods for variational inequalities. Springer, Berlin (2000).
\bibitem{K1976} Korpelevich, G. M., The extragradient method for finding saddle points and other problems, Ekonomikai Matematicheskie Metody. 
12, 747-756 (1976).
\bibitem{MO1992} Muu, L.D., Oettli, W., Convergence of an adative penalty scheme for finding constrained equilibria. Nonlinear
Anal. TMA, 18 (12), 1159-1166 (1992).
\bibitem{MS2015}Malitsky, Yu. V., Semenov, V. V., A hybrid method without extrapolation step for solving variational inequality problems. 
J. Glob. Optim., 61, 193-202 (2015).
\bibitem{M1970}Martinet, B., R$\rm\acute{e}$gularisation d$\rm\acute{}$ in$\rm\acute{e}$quations variationelles par approximations successives. 
Rev. Fr. Autom. Inform. Rech. Op$\rm\acute{e}$r., Anal. Num$\acute{e}$r., 4, 154--159 (1970).
\bibitem{MPPP2012} Mordukhovich, B., Panicucci, B., Pappalardo, M., Passacantando, M., Hybrid proximal methods for
equilibrium problems. Optim. Lett., 6, 1535-1550 (2012).
\bibitem{VSH2013}Nguyen, T. T. V., Strodiot, J. J., Nguyen, V. H., Hybrid methods for solving simultaneously
an equilibrium problem and countably many fixed point problems in a Hilbert space. J. Optim. Theory Appl., (2013). DOI 10.1007/s10957-013-0400-y.
\bibitem{QMH2008}Quoc, T.D., Muu,  L.D., Hien, N.V., Extragradient algorithms extended to equilibrium problems. Optimization, 57, 749-776 (2008).
\bibitem{R1970} Rockafellar, R.T., Convex Analysis. Princeton, NJ: Princeton University Press, 1970.
\bibitem{R1976}Rockafellar, R.T., Monotone operators and the proximal point algorithm. SIAM J. Control Optim., 14, 877--898 (1976).
\bibitem{SS2000}Solodov, M. V., Svaiter, B. F., Forcing strong convergence of proximal point iterations in Hilbert space. Math. Program., 
87, 189-202 (2000).
\bibitem{SVH2011}Strodiot, J. J., Nguyen, T. T. V, Nguyen, V. H., A new class of hybrid extragradient algorithms
for solving quasi-equilibrium problems. J. Glob. Optim., 56 (2), 373-397 (2013).
\bibitem{TT2007}Takahashi, S., Takahashi, W., Viscosity approximation methods for equilibrium problems and fixed point in 
Hilbert space. J. Math. Anal. Appl., 331 (1), 506-515 (2007).
\bibitem{TandaTaka2007} Tada, A., Takahashi, W., Weak and strong convergence theorems for a nonexpansive mapping and an equilibrium problem. 
J. Optim.Theory Appl., 133, 359--370 (2007).
\bibitem{TT2003} Takahashi, W., Toyoda, M., Weak convergence theorems for nonexpansive mappings and
monotone mappings. J. Optim. Theory Appl., 118, 417--428 (2003).
\bibitem{JK2010} Jaiboon, C., Kumam, P., Strong convergence theorems for solving equilibrium problems and
fixed point problems of $\xi$ -strict pseudo-contraction mappings by two
hybrid projection methods. J. Comput. Appl. Math., 234, 722-732 (2010).
\end{thebibliography}
\end{document}